\documentclass[12pt]{amsart}

\usepackage{dsfont, marginnote,geometry}
\usepackage{amsmath,mathrsfs,amsthm, amsaddr}
\usepackage[usenames,dvipsnames,svgnames,table]{xcolor}
\usepackage{cancel,amssymb}
\usepackage[cyr]{aeguill}
\usepackage{graphicx}
\usepackage[pdfdisplaydoctitle=true,
            colorlinks=true,
            urlcolor=blue,
            citecolor=blue,
            linkcolor=blue,
            pdfstartview=FitH,
            pdfpagemode= UseNone,
            bookmarksnumbered=true]{hyperref}
\usepackage[color=blue!20, linecolor=red!]{todonotes}
\usepackage{lineno}

\newtheorem{theorem}{Theorem}[section]

\theoremstyle{definition}

\theoremstyle{remark}

\numberwithin{equation}{section}

\newcommand{\RR}{\mathbb{R}} 

\newcommand{\Ss}[1]{W^{k,p}(\mathbb{S}^1)}

\newcommand{\abs}[1]{\left\vert#1\right\vert}

\DeclareMathOperator{\cl}{cl}

\title{The pressure distribution in extreme Stokes waves}

\begin{document}
\author{Tony Lyons}
\address{Department of Computing and Mathematics\\ Waterford Institute of Technology\\ Waterford, Ireland}
\email{tlyons@wit.ie}
\keywords{Euler's Equation, Weak Solutions, Maximum Principles}
\subjclass[2010]{35Q31, 35Q35, 76B99}
\begin{abstract}
 In this paper we prove that the pressure beneath an extreme Stokes wave over finite depth is strictly increasing with depth. Additionally it is shown that the pressure decreases in moving between a crest-line and trough-line, while it is stationary with respect to the horizontal coordinate along these lines themselves.
\end{abstract}
\maketitle
\section{Introduction}\label{Sec1}
In the current paper we investigate the behaviour of the pressure profile induced by a Stokes wave of greatest height propagating in a fluid body of finite-depth. A Stokes waves of greatest height, also called an extreme Stokes waves, is a steady, periodic, two-dimensional flow, characterised by the appearance of a stagnation point on the free surface, specifically at the wave crest, which appear as cusps with an included angle of $120^{\circ}$ cf. \cite{Sto1880}. The symmetries and regularity of steady periodic waves have been well established under a variety of conditions, including gravity waves incorporating vorticity cf. \cite{Con2011, CE2004a, CE2004b, CE2011b} . However the presence of  stagnation points in the flow introduces a number of mathematical complications related to the regularity of the of the free surface \cite{ Con2012, Tol1996}, which require careful and subtle treatment when analysing the governing equations of the flows. The analytical techniques developed in this paper by means of an excision of the stagnation points allow us to deduce several key features of the pressure in the fluid flow. Specifically it is shown that the pressure decreases as we move horizontally between a crest-line and a trough-line, while along these lines the pressure is stationary with respect to the horizontal coordinate. Secondly it will be shown that the pressure strictly increases with depth at any point throughout the fluid domain.

Incompressibility and irrotationality of regular Stokes waves allows one to construct a conformal map of the fluid domain with analytic continuation to the boundary, the so called hodograph transform. Using the hodograph transform one may equivalently reformulate the free-boundary problem governing two-dimensional Stokes waves as a fixed-boundary problem on the conformal image of the fluid domain, wherein one may readily employ several powerful maximum principles. However the presence of a stagnation point in an extreme Stokes wave ensures that the hodograph transform does not have analytic continuation to the boundary; as such the fixed-boundary and free-boundary problems may no longer be identified on the boundary. Moreover, owing to the lack of regularity of the free surface, the boundary of the conformal domain lacks the necessary regularity to impose maximum principles. Nevertheless by excising the stagnation point from the conformal domain, we restore the necessary regularity of the boundary to impose maximum principles and also ensure the hodograph transform is conformal throughout. Then by means of a uniform limiting process cf. \cite{AFT1982} we may extend the application of maximum principles from the excised domain to the fluid domain. The uniform limiting process was used to great effect in the work \cite{AFT1982}, which proved the existence of extreme deep-water Stokes incorporating cusps with an included angle of 120 degrees. Moreover, the same technique was employed in the  work \cite{PT2004} to prove the convexity of the surface profile of Stokes waves of greatest height. In the recent works \cite{Con2012} and \cite{Lyo2014}, the uniform limiting process was effectively used to show that particles travelling on the streamlines of extremes Stokes waves, in finite and infinite depth, do not form closed trajectories.

Extending the methods developed in \cite{Con2012} and \cite{Lyo2014}, we demonstrate the aforementioned qualitative behaviour of the pressure in an extreme Stokes wave in fluid body of finite depth; see \cite{Hen2008, Hen2011, Lyo2015} for recent results concerning the pressure in deep-water Stokes waves. Pressure measurements within a water body are of significant importance both experimentally and theoretically. In many respects the pressure determines several physical features of the fluid body, for example governing the trajectories followed by individual fluid particles, see \cite{Con2006,Con2012, CE2007, Hen2006, Hen2008, Lyo2014} for some recent results on particle trajectories in regular and extreme Stokes waves. From a practical point of view, wave measurements can be extremely difficult to execute \cite{HKS2008,Ume2005} and a widely used method of obtaining information is through pressure data \cite{OVDH2012,THYLL2005}.  These pressure measurements convey information about the surface profile of the water wave by means of the pressure transfer function c.f. \cite{Con2012a, CEH2011, ES2008, CC2013, Hen2013} for some recent results in this regard. From an experimental and commercial perspective, information about the velocity and pressure beneath surface waves are crucial to the design of cost effective off-shore structures \cite{DKD1992}.

In Section \ref{Sec2} we introduce the free boundary problem for extreme Stokes waves in the moving frame, that is to say the frame which moves with wave phase speed of the two-dimensional flow. Section \ref{Sec3} examines the stream function and velocity potential for the flow, which follow from the incompressibility and irrotationality of the flow, and introduce the hodograph transform defined in terms of these functions. The free boundary problem is reformulated as a fixed boundary problem on the conformal image of the fluid domain under the hodograph transform. In Section \ref{Sec4} we prove the main theorem of this paper. Theorem \ref{s4thm1} states that the fluid pressure is strictly decreasing as a function of the horizontal coordinate $x$ between a crest-line and subsequent trough-line, except along these lines themselves, where the pressure is stationary with respect to this $x$-coordinate. Additionally it is  proven that the pressure is a strictly decreasing function of the vertical coordinate $y$, that is the pressure increases with the fluid depth.

\section{Physical Considerations}\label{Sec2}
\subsection{The Free Boundary Problem}\label{Sec2.1}
There are several structural features common to all Stokes waves, which are crucial for the analysis of the flow. Stokes waves are steady, irrotational and periodic wave trains, moving with fixed wave speed $c$ relative to the physical frame with coordinates $(X,Y,Z)$. In a Stokes wave of greatest height, the free surface profile $\eta:=\eta(X-ct)$ is symmetric about any crest-line and moreover is strictly decreasing between wave-crest and wave-trough cf. \cite{Tol1996}. In general for Stokes waves, the horizontal fluid  velocity $u$ and pressure $P$ are periodic and symmetric about the crest-line while the vertical velocity $v$ is periodic and antisymmetric about the crest-line cf. \cite{Con2011}. Without loss of generality it will be assumed that the wave has a fixed wavelength $\lambda =2\pi$, defined as the distance between two successive wave crests (or troughs).  Stokes waves are characterised as two-dimensional flows in a vertical plane where the effects of surface tension are negligible, and which move through an incompressible fluid body.

As with regular Stokes waves, the wave of greatest height is a steady wave form and accordingly it propagates with fixed wave speed $c$. As such, we may reformulate the free boundary problem in the moving frame, which has coordinates $(x,y,z)$, and is related to the physical frame coordinates by the transformation
\begin{equation}\label{s2eq1}
    x = X-ct,\quad y=Y,\quad z=Z.
\end{equation}
The governing equations in the moving frame are given by Euler's equation
\begin{equation}\label{s2eq2}
\begin{split}
\left.
\begin{split}
    &(u-c)u_x + +vu_y=-P_x\\
    &(u-c)v_x + vv_y = -P_y-g,
\end{split}
\right.
\end{split}
\end{equation}
where $u:=u(x,y)$, $v:=v(x,y)$, $P:=P(x,y)$ and $\eta:=\eta(x,y)$ and the fluid domain is given by the closure of the set $\Omega = \left\{(x,y):x\in\RR,y\in(-d,\eta(x))\right\}$. The reformulation of the Euler equations with respect to the moving frame coordinates is that the system \eqref{s2eq2} depends on time implicity through the $x$-variable. Thus the velocity field and pressure depend on time implicity and are of the form $(u,v):=(u(x,y),v(x,y))$ and $P:=P(x,y)$. The incompressibility and irrotationality conditions are expressed via
\begin{equation}\label{s2eq3}
    u_x+v_y = 0\quad\text{and}\quad u_y-v_x = 0
\end{equation}
in $\Omega$. The dynamic and kinematic boundary conditions are given by
\begin{equation}\label{s2eq4}
\begin{split}
&P=P_0\quad\text{and}\quad v= (u-c)\eta^{\prime}\qquad\text{on }y=\eta\\
&v= 0\qquad\text{on }y=-d.
\end{split}
\end{equation}
The system of equations \eqref{s2eq2}-\eqref{s2eq4} constitute the free boundary problem in the moving frame.

The incompressibility and irrotationality of the fluid flow, as per equation \eqref{s2eq3}, ensure that both components of the fluid velocity are harmonic in the fluid domain, namely
\begin{equation}\label{s2eq5}
\Delta u = 0\quad\text{and}\quad \Delta v = 0\qquad\text{for }(x,y)\in \Omega.
\end{equation}
The dynamic boundary conditions in equation \eqref{s2eq4} reflect the fact that the only pressure exerted on the fluid particles there is the constant atmospheric pressure $P_0$. Meanwhile the first kinematic boundary condition coincides with the fact particles on the free surface remain there while the second kinematic boundary condition reflects the impermeability the flat bed.

\subsection{Stagnation Points and Mean Currents}
In the case of regular Stokes waves, the horizontal velocity component in the physical frame of an individual fluid particle is always less than the wave speed, which yields the strict inequality
\begin{equation}\label{s2eq6}
    u(x,y)-c<0\quad\text{for all }(x,y)\in\overline{\Omega}.
\end{equation}
in the moving frame. Hence the horizontal component of a particle's velocity is in the negative $x$-direction within the moving frame. The presence of stagnation points at the wave crest in a Stokes wave of greatest height means the relation \eqref{s2eq6} is no longer strict throughout the fluid body the, but rather we find
\begin{equation}\label{s2eq7}
    u(x,y)-c\leq 0\quad\text{for all }(x,y)\in\overline{\Omega},
\end{equation}
with equality achieved only at the wave-crest. The presence of stagnation points on the free surface will ensure that the functions $u$, $v$, $\eta$ and $P$ are no longer real analytic, but merely continuous, in any neighbourhood of these points. However these functions are real analytic throughout the remainder of the fluid domain in the wave of greatest height, and indeed are also real analytic throughout the entire fluid domain of regular Stokes waves. This regularity of the velocity, pressure and surface profile in a smooth Stokes waves will be exploited to deduce several qualitative features of these functions in the wave of greatest height, by means of the uniform limiting process.

In the case of regular and extreme Stokes waves over finite depth, we define the mean horizontal current on the sea bed as
\begin{equation}\label{s2eq8}
    \kappa = \frac{1}{2\pi}\int_{0}^{2\pi} u(x,-d)dx < c,
\end{equation}
with the latter inequality a result of \eqref{s2eq7}, having imposed the condition that equality is achieved only at the wave crest. In addition the irrotationality of the fluid flow also ensures that at any fixed depth $y=y_0$ below the free surface we also find
\begin{equation}\label{s2eq9}
    \kappa = \frac{1}{2\pi}\int_{0}^{2\pi}u(x,y_0)dx < c,
\end{equation}
the last inequality again assured by the fact there are no stagnation points beneath the free surface. Thus we find the average horizontal current is fixed for all depths below the free surface of the wave of greatest height, and in what follows we will assume $\kappa = 0$.

\section{The Hodograph Transform}\label{Sec3}

The incompressibility of the fluid flow provides for the existence of streamlines in the fluid domain, with each streamline composed of the same particles at all times. This foliation of the fluid domain is defined by the level sets of the stream function $\psi:=\psi(x,y)$, which is defined in terms of the velocity field as follows:
\begin{equation}\label{s3eq2}
\psi_x = -v\quad\text{and}\quad\psi_y= u-c\qquad\text{for }(x,y)\in\Omega.
\end{equation}
The incompressibility condition applied to equation \eqref{s3eq2} also leads us to conclude that the stream function is harmonic throughout the fluid domain $\Omega$ and as such we must have $\Delta\psi = 0$ therein. Integrating the relations \eqref{s3eq2} we deduce from the symmetry properties and periodicity of the velocity field that the stream function itself is periodic in the horizontal direction.

Analogously the irrotationality of the fluid flow provides for the existence of a  velocity potential $\phi:=\phi(x,y)$, which is defined in terms of the fluid velocity as follows:
\begin{equation}\label{s3eq3}
\left.
\begin{aligned}
    &\phi_x = u-c\\
    &\phi_y = v\\
\end{aligned}
\right\}\quad\text{for }(x,y)\in \Omega.
\end{equation}
Once again the irrotationality condition applied to \eqref{s3eq3} also ensures the velocity potential $\phi$ is harmonic throughout the fluid domain $\Omega$. Furthermore upon integrating \eqref{s3eq3} and imposing the symmetries and periodicity of the fluid velocity we deduce that the function $\phi(x,y)+cx$ is also periodic with respect to $x$.

Integrating the Euler equation \eqref{s2eq1}, it is readily demonstrated that the expression
\begin{equation}\label{s3eq4}
\frac{(u-c)^2+v^2}{2}+g(y-d) + P
\end{equation}
is constant throughout the fluid domain, which constitutes Bernoulli's law. Utilising this it is possible to reformulate the free boundary problem in equations \eqref{s2eq1}--\eqref{s2eq4} in terms of the stream function $\psi$ and the free surface $\eta$ according to:
\begin{equation}\label{s3eq5}
\begin{split}
    &\Delta\psi = 0\quad \text{on } (x,y)\in \Omega\\
&\left.
\begin{split}
    \abs{\nabla\psi}^2 + 2g(y+d) = Q\\
    \psi = 0
\end{split}
\right\}\quad \text{ on } y=\eta(x)\\
&\psi=m\quad \text{on } y=-d,
\end{split}
\end{equation}
with $m$ the relative mass flux and $Q$ the hydraulic head.

\subsection{The Conformal Map of the Fluid Domain}
Having constructed a pair of harmonic conjugate functions, namely $\phi$ and $\psi$, we are now in a position to use these maps to transform the free boundary domain to a fixed boundary domain under the conformal map $\mathcal{H}: \Omega\to\hat{\Omega}$ where we define $(q,p)\in\cl\hat{\Omega}$ as follows:
\begin{equation}\label{s3eq6}
\begin{split}
    &q=-\phi(x,y)\\
    &p=-\psi(x,y).
\end{split}
\end{equation}
It should be noted that the mapping in equation \eqref{s3eq6} is conformal almost every where in the fluid body except at the wave-crest where the velocity field is no longer real analytic.

In what follows we shall apply maximum principles to the velocity field and pressure in the conformal domain $\hat{\Omega}$. However owing to the periodicity of these functions it is only necessary that they be examined over one wavelength, namely between $x=-\pi$ and $x=\pi$ in the free boundary domain. As such the interior domain may be effectively represented by the two disjoint regions
\begin{equation}\label{s3eq7}
\begin{split}
    &\Omega_+ = \left\{(x,y):x\in(0,\pi),y\in(-d,\eta(x))\right\}\\
    &\Omega_- = \left\{(x,y):x\in(-\pi,0),y\in(-d,\eta(x))\right\},
\end{split}
\end{equation}
along with their lateral edges
\begin{equation}
    \left\{x=0:y\in[-d,\eta(0)]\right\}\quad\text{and}\quad\left\{x=\pm\pi:y\in[-d,\eta(\pm\pi)]\right\}
\end{equation}
corresponding to the closure of the
crest-line and trough-lines respectively.
Meanwhile the free surface is divided as follows:
\begin{equation}\label{s3eq8}
\begin{split}
    &S_+ = \left\{(x,y): x\in(0,\pi),y=\eta(x)\right\}\\
    &S_- = \left\{(x,y): x\in(-\pi,0),y=\eta(x)\right\},
\end{split}
\end{equation}
while the lower boundary of the fluid domain is given by the sets
\begin{equation}\label{s3eq9}
\begin{split}
    &B_+ = \left\{(x,y): x\in(0,\pi),y=-d\right\}\\
    &B_- = \left\{(x,y): x\in(-\pi,0),y=-d\right\},
\end{split}
\end{equation}
corresponding to the impermeable flat bed.

Under the conformal map $\mathcal{H}$ the interior regions ${\Omega}_{\pm}$ transform to the interior regions $\hat{\Omega}_{\pm}$ given by
\begin{equation}\label{s3eq10}
\begin{split}
    &\hat{\Omega}_+ = \left\{(q,p):q\in(0,c\pi),p\in(-m,0)\right\}\\
    &\hat{\Omega}_- = \left\{(q,p):q\in(-c\pi,0),p\in(-m,0)\right\}.
\end{split}
\end{equation}
Meanwhile the images of the free surface and the flat bed are
\begin{equation}\label{s3eq11}
\begin{split}
    &\hat{S}_+ = \left\{(q,p): q\in(0,c\pi),p=0\right\}\\
    &\hat{S}_- = \left\{(q,p): q\in(-c\pi,0),p=0\right\}
\end{split}
\end{equation}
and
\begin{equation}\label{s3eq12}
\begin{split}
    &\hat{B}_+ = \left\{(q,p): q\in(0,c\pi),p=-m\right\}\\
    &\hat{B}_- = \left\{(q,p): q\in(-c\pi,0),p=-m\right\}
\end{split}
\end{equation}
respectively. The remaining boundary lines are given by the image of the closed crest-line, that is $\left\{q=0:p\in[-m,0]\right\}$ and the images of the closed trough-lines, namely $\left\{q=\pm c\pi:p\in[-m,0]\right\}$.

\begin{figure}
\centering
\includegraphics{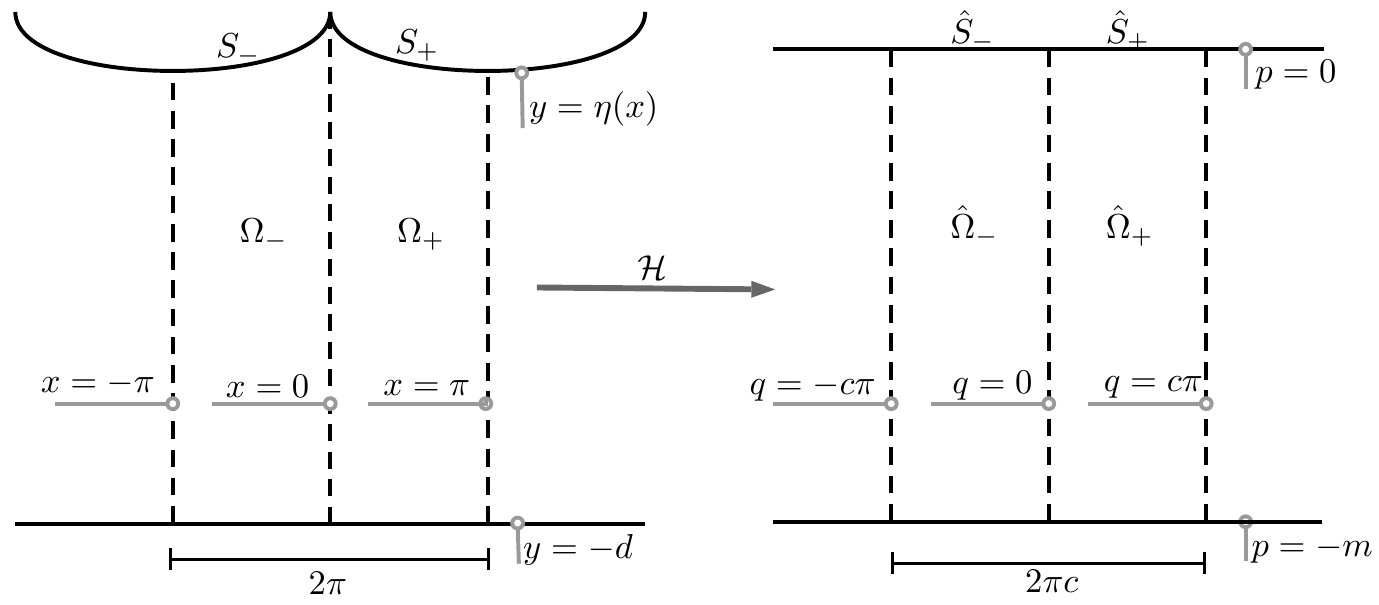}
\caption{The hodograph transform between the free boundary domain ${\Omega}$ and its conformal image $\hat{\Omega}$.}
\label{hodograph}
\end{figure}

Lastly we introduce the height function
\begin{equation}\label{s3eq13}
    h(q,p)=y+d,
\end{equation}
where it is immediately clear that this function is harmonic throughout $\Omega$. Since the mapping given by equation \eqref{s3eq8} is conformal every where within $\Omega$, it follows that $h(q,p)$ is also analytic everywhere within the conformal domain $\hat{\Omega}$. The free boundary problem given by equation \eqref{s3eq5} may be reformulated in terms of the harmonic function $h$ in the conformal domain, thus yielding a fixed boundary problem given by:
\begin{equation}\label{s3eq14}
\begin{split}
&\Delta_{q,p}h=0\quad \text{for all }(q,p)\in\hat{\Omega}\\
&h=0\qquad\text{on }p=-m\\
&2(Q-gh)(h_q^2+h_p^2)\quad \text{on }p=0.
\end{split}
\end{equation}
The second boundary condition (on $p=0$) follows from Bernoulli's Law and the coordinate transformations
\begin{equation}\label{s3eq15}
\begin{split}
&\left.
\begin{split}
    &\partial_x = (c-u)\partial_q+v\partial_p\\
    &\partial_y = -v\partial_q + (c-u)\partial_p
\end{split}
\right\}\qquad
\left.
\begin{split}
    &\partial_q = h_p\partial_x + h_q\partial_y\\
    &\partial_p= -h_q\partial_x + h_p\partial_y
\end{split}
\right\}\\
&h_q = -\frac{v}{(c-u)^2+v^2}\qquad h_p=\frac{c-u}{(c-u)^2+v^2},
\end{split}
\end{equation}
which follow directly from the conformal transformation in equation \eqref{s3eq6} and equations \eqref{s3eq2}--\eqref{s3eq3} which define the harmonic conjugates $\phi$ and $\psi$. In a regular Stokes wave the reformulation of the free boundary problem as per equation \eqref{s3eq14} is valid throughout the fluid domain including the free boundary. In the case of extreme Stokes waves, the conformal change of variables given by equation \eqref{s3eq6} may be implemented in the interior, and by Caratheodory's theorem c.f. \cite{Pom2002} may be continuously extended to the closed fluid domain. Nevertheless the coordinate transformation is not differentiable at the wave crest and as such the reformulation of the boundary problem fails at the point $(q,p)=(0,0)$ along the boundary $\partial\hat{\Omega}$.

The map given by $\gamma:\hat{\Omega}_+\to\Omega_+$ with $\gamma(\xi) = x+iy$ and $\xi=q+ip$ is analytic in the interior of $\hat{\Omega}_+$. Additionally the function has analytic continuation to a neighbourhood of any point on the boundary $\partial{\hat{\Omega}_+}.$ Defining the function $h_0(q) = h(q,0)$, there exists a sequence of Stokes waves $\left\{h_n(q):n\geq1\right\}$ which converges weakly to $h_0$ in the Sobolev space $W_{\mathrm{per}}^{1,k}(\RR)$, where $1< k< 3$ cf. \cite{BT2003}. Privalov's theorem ensures the  sequence $x+iy_n = \gamma_n$ for $n\geq1$ approaches the extreme Stokes wave $h_0(q)$. The uniform limiting process is the means by which the sequence of regular almost extreme Stokes waves $\left\{h_n:n\geq 1\right\}$ allows us to formulate maximum principle in the conformal domain for the functions $u$, $v$ and $P$ and by taking the limit of the sequence, extend our conclusions to the extreme Stokes wave $h_0$.

\section{The Pressure Distribution}\label{Sec4}
In this section we will prove the main result of this paper which is the following:
\begin{theorem}\label{s4thm1}
In the fluid domain $\overline{\Omega}$ the pressure distribution function possesses the following features
\begin{enumerate}
\item The pressure strictly decreases between the crest-line and trough-line and as such we have $P_x<0$, except on the crest-line and on the trough-lines themselves where we find $P_x=0$.
\item The pressure strictly increases as we descend into the fluid domain and as such $P_y<0$
      throughout $\overline{\Omega}$.
\end{enumerate}
\end{theorem}

\begin{proof}
\textsc{Part 1.}
Given that $v$ is $2\pi$-periodic in $x$ and anti-symmetric about the crest-line $\left\{x=0:y\in[-d,\eta(0)]\right\}$ (properties inherited from the almost extreme Stokes waves via the uniform limiting process),  it follows that
\begin{equation}\label{s4eq1}
\begin{split}
   v(0,y) = v_y(0,y)= 0 & \text{ for } -d<y<\eta(0)  \\
    v(0,y) = v_{y}(\pi,y)=0 & \text{ for }-d<y<\eta(\pi).
\end{split}
\end{equation}
Moreover, owing to the incompressibility condition in the interior of $\Omega$ cf. equation \eqref{s2eq3}, it follows that
\begin{equation}\label{s4eq2}
\begin{split}
   u_x(0,y) =0 & \text{ for } -d<y<\eta(0)  \\
   u_{x}(\pi,y) =0 & \text{ for }-d<y<\eta(\pi).
\end{split}
\end{equation}
Meanwhile, the Euler equation \eqref{s2eq2} along with equations \eqref{s4eq1}--\eqref{s4eq2} yield
\begin{equation}\label{s4eq3}
  P_{x} = (c-u)u_{x} = 0 \quad \text{for } \left\{(x,y):x \in\{0,\pi\} \text{ and } y\in(-d,\eta(x))\right\}.
\end{equation}
This ensures $P$ is stationary with respect to the horizontal coordinate along the crest(trough)-line, from a point immediately beneath the wave-crest(trough) to points directly beneath lying immediately above the sea bed.

The velocity component $u(x,y)$ is continuous throughout $\overline{\Omega}$ cf. \cite{Con2012}. As such, about the wave crest $(0,\eta(0))$ we may write
\begin{equation}\label{s4eq4}
    u_x(0,\eta(0)) = \lim_{\epsilon\to0}\frac{u(\epsilon,\eta(0))}{\epsilon} = -\lim_{\epsilon\to0} \frac{u(-\epsilon,\eta(0))}{\epsilon} = 0,
\end{equation}
where the  symmetry of  $u$ about the crest-line $\left\{x=0:y\in[-d,\eta(0)]\right\}$ ensures the last inequality. Moreover anti-symmetry ensures $v$ vanishes at the wave-crest, and as such the Euler equation \eqref{s2eq2} yields
\begin{equation}\label{s4eq5}
    P_x = (c-u)u_x = 0\quad \text{at }(x,y) = (0,\eta(0))
\end{equation}
thus ensuring $P$ is stationary with respect to $x$ at the wave-crest also. Meanwhile, at the wave-trough $(\pi,\eta(\pi))$ the functions $u,$ $v$ and $P$ are differentiable and we may write
\begin{equation}\label{s4eq6}
    P_x = (c-u)u_x = (c-u)v_y = 0 \quad\text{at  }(x,y)  = (\pi,\eta(\pi)),
\end{equation}
having imposed $v(\pi,\eta(\pi)) = 0$ and equation \eqref{s4eq1}.

Additionally along the flat-bed $y=-d$ where $x\in\RR$, we have $v=0$ due to the impermeability of the sea-bed. Moreover the functions $u$, $v$ and $P$ are differentiable along the flat-bed, in which case equations \eqref{s2eq1}--\eqref{s2eq5} yield
\begin{equation}\label{s4eq7}
    P_x = (u-c)v_y.
\end{equation}
With $x\in(0,\pi)$ we have $v_y(x,-d) > 0$, cf. \cite{Con2012}, while $c>u(x,-d)$ and it follows that
\begin{equation}\label{s4eq8}
    P_x <0\quad \text{for }x\in(0,\pi)\ y=-d.
\end{equation}
Meanwhile periodicity and anti-symmetry of $v$ with respect to $x$, also ensure $v(0,-d) = v(\pi,-d) = 0$, in which case equations \eqref{s2eq3} and \eqref{s4eq1} allow us to write
\begin{equation}\label{s4eq9}
 P_x = (c-u)u_x = (u-c)v_y = 0\quad\text{for } x\in\{0,\pi\}\text{ and }y=-d,
\end{equation}
where we use the fact there are no stagnation points on the flat bed.

Within $\Omega_+$ and along $S_+$ it has been shown that $u_q<0$ cf. \cite{Con2012}, while equations \eqref{s2eq2} and \eqref{s3eq15} ensure
\begin{equation}\label{s4eq10}
 P_x = \frac{u_q}{h_q^2+h_p^2} \quad\text{in }\Omega_{+} \text{ and along } S_+.
\end{equation}
It follows that $P_x<0$ in $\overline{\Omega}_+$ except along the crest-line and trough-line where $P$ is stationary with respect to the  $x$-coordinate.

\textsc{Part 2}\indent Next we consider the function $f(x,y):=(c-u(x,y))v(x,y)-gx$ defined on $\overline{\Omega}$. It may be shown that
\begin{equation}\label{s4eq11}
\frac{d}{dx}f(x,\eta(x)) \leq 0 \quad \text{for }x\in[0,\pi],
\end{equation}
see reference \cite{CS2010} for a derivation of this inequality in regular Stokes waves, where the results of \cite{Ami1987,PT2004, Con2012, Lyo2015} also ensure this derivation may be extended to the case of extreme Stokes waves.
Integrating \eqref{s4eq11} it follows at once that $f<0$ on $S_+$, while antisymmetry and $2\pi$-periodicity ensure $f=0$ along $\{x=0:y\in[-d,\eta(0)]\}$ and $f=-g\pi$ along $\{x=\pi:y\in[-d,\eta(\pi)]\}$. Since the sea-bed is impermeable and $g>0$ we also deduce $f<0$ on $B_+$. Moreover, since the hodograph transform maps the boundary to the boundary, namely $\mathcal{H}:\partial\Omega\to\partial\hat{\Omega}_+$ we conclude $f=0$ on $\left\{q=0:p\in[-m,0]\right\}$ while $f<0$ elsewhere along the boundary $\partial\hat{\Omega}_+$.

Since the velocity field $(u,v)$ is continuous on $\overline{\Omega}$ c.f. \cite{Con2012}, the function $f$ must be also continuous in $\overline{\Omega}_+$.  Direct calculation also reveals $f$ to be harmonic in $\Omega$, and since the mapping $\mathcal{H}:\Omega\to\hat{\Omega}$ is conformal in the interior, $f$ is also harmonic in $\hat{\Omega}_+$. The previous paragraph has also shown $f\leq0$ on $\partial\hat{\Omega}_+$, in which case the weak maximum principle requires $f$ achieve its maximum value on $\partial\Omega_+$, and thus $f\leq0$ in $\hat{\Omega}_+$. Consequently we must also have $f\leq0$ in $\Omega$.

In fact it may be shown that $f<0$ in $\Omega_+.$ Let us consider the excised conformal domain $\hat{\Omega}_{+}^{\varepsilon}$, illustrated in Figure \ref{excisedregion}, where a quarter disc of radius $\varepsilon$ centered at $(q,p)=(0,0)$ is removed. Having removed the stagnation point $(0,0)\in\partial\hat{\Omega}_{+}$, the functions $u$, $v$ and $f$, along with the boundary $\partial\hat{\Omega}_{+}^{\varepsilon}$, possess the necessary regularity to impose the strong maximum principle. If there exists an interior point $(q_0,p_0)\in\hat{\Omega}_{+}^{\varepsilon}$ such that $f(q_0,p_0)=0$, then $f$ has achieved its maximum value in the interior $\hat{\Omega}_{+}^{\varepsilon}$, which requires $f$ be constant.  However $f$ is not constant and so we conclude $f<0$ in $\hat{\Omega}_{+}^{\varepsilon}.$   Moreover since the hodograph transform $\mathcal{H}:{\Omega}_+^{\varepsilon}\to{\hat{\Omega}}_{+}^{\varepsilon}$ is conformal, it follows that $f<0$ in $\Omega_{+}^{\varepsilon}$. The uniform limiting process then ensures $f<0$ in $\Omega_+$.
\begin{figure}
\centering
\includegraphics{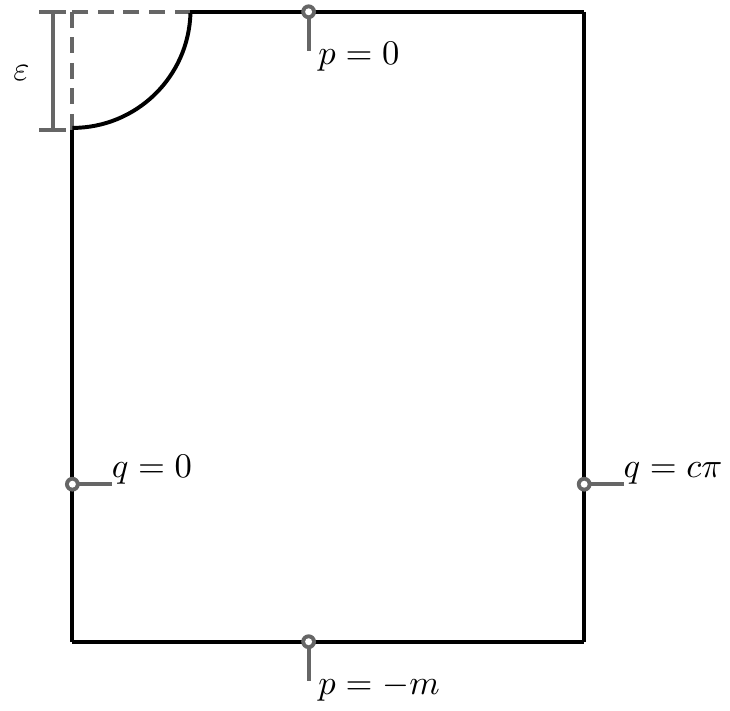}
\caption{The excised conformal domain $\hat{\Omega}_{+}^{\varepsilon}$.}\label{excisedregion}
\end{figure}
Since $f=0$ along the crest line and $f<0$ in  the interior $\Omega_+$,we deduce that $f_x(0,y)<0$ for $y\in(-d,\eta(0))$. Along  the crest-line, the antisymmetry of $v$ allows us to reformulate equation \eqref{s2eq1} according to
\begin{equation}\label{s4eq12}
    P_y(0,y)=f_x(0,y)<0\quad\text{for }y\in(-d,\eta(0)).
\end{equation}
With $v=0$ along the trough-line and $v>0$ in the interior $\Omega_+$ cf. \cite{Con2012}, it follows that $v_x(\pi,y)<0$ for $-d<y<\eta(\pi)$ and consequently the irrotationality condition given by equation \eqref{s2eq3} ensures $u_y(\pi,y)<0$ for $-d<y <\eta(\pi)$. Thus the second member of the Euler equation \eqref{s2eq2} yields
\begin{equation}\label{s4eq13}
    P_y(\pi,y) = -g+ (c-u(\pi,y))u_y(\pi,y) < 0\quad\text{for }y\in\left(-d,\eta(\pi)\right),
\end{equation}
since $c-u>0$ everywhere except the stagnation point.

It may be deduced from equation \eqref{s2eq2} that
\begin{equation}\label{s4eq14}
  \Delta P = -2(u_x^2+u_y^2)<0 \quad\text{in } \Omega,
\end{equation}
in which case $P$ is super-harmonic and consequently
\[\min_{\overline{\Omega}}P = \min_{\partial\Omega}P.\]
The second kinematic boundary condition in equation \eqref{s2eq4} ensures $v(x,-d)= v_x(x,-d) = 0$ for $x\in\RR$  which combined with equation \eqref{s2eq2} yields
\begin{equation}\label{s4eq15}
    P_y(x,-d) = -g<0 \quad \text{for }x\in\RR.
\end{equation}
Hence the minimum of $P$ must be achieved along the free surface, and since $P(x,\eta(x)) = P_0$ for all $x\in\RR$, it follows that $P$ attains its minimum all along the free surface and as such Hopf's maximum principle ensures
 \begin{equation}\label{s4eq16}
   P_y(x,\eta(x))<0\quad\text{for }x\in\RR.
 \end{equation}
At this point we have shown $P$ is strictly increasing with depth at any point along the boundary $\partial\Omega_+$.

However the coordinate transformation in equation \eqref{s3eq15} yields
\begin{equation}\label{s4eq17}
f_q = v_x-\frac{g(c-u)}{(c-u)^2 + v^2}\quad \text{in } \hat{\Omega}_+,
\end{equation}
which combined with the impermeability of the flat bed ensures $v=v_x=0$ along $\hat{B}_+$ ensures
\begin{equation}\label{s4eq18}
    f_q = -\frac{g}{c-u}<0 \quad\text{on }\hat{B}_+,
\end{equation}
since $c-u>0$ in the conformal domain except at  $(q,p)=(0,0)$. Along the crest-line  we have $v=0$ which combined with the Euler equation yields
\begin{equation}\label{s4eq19}
    f_q(0,p) = \frac{P_y}{c-u}=-\frac{g}{c-u}<0\quad\text{for }\text{ and }-m\leq p< 0,
\end{equation}
while along the trough-line we also find $v=0$, which yields
\begin{equation}\label{s4eq20}
    f_q(c\pi,p) = \frac{P_y}{c-u}=-\frac{g}{c-u}<0\quad\text{for }\text{ and }-m\leq p\leq 0,
\end{equation}
having used $c-u>0$ in the closure of $\hat{\Omega}_+$, except at the stagnation point corresponding to $(q,p)=(0,0)$. Along the free surface segment $\left\{y=\eta(x)\colon 0<x\leq \pi\right\}$, whose image is $\{p=0\colon 0<q\leq c\pi]\}$ under the hodograph transform, we find
\begin{equation}\label{s4eq21}
    f_q = \frac{\frac{d}{dx}f(x,\eta(x))}{(c-u)(1+\eta^{\prime\,2})}\leq0 ,
\end{equation}
having used relation \eqref{s4eq11} to obtain the last inequality. Thus we have shown $f_{q}\leq0$ along $\partial\hat{\Omega}_{+}^{\varepsilon}$.

 Since $f$ is harmonic in $\Omega_+$ and the hodograph transform $\mathcal{H}$ is conformal within $\Omega_+$, it follows that $f$ and consequently $f_q$ are also harmonic on the open domain $\hat{\Omega}_{+}^{\varepsilon}$. As with our previous considerations regarding $f$, maximum principles applied to $f_q$  ensure $f_q<0$ within the domain $\hat{\Omega}_+^{\varepsilon}$, having used the result of the previous paragraph. The uniform limiting process then ensures $f_q<0$ throughout the open rectangle $\hat{\Omega}_+$. Meanwhile equation \eqref{s4eq11} now ensures
\begin{equation}\label{s4eq22}
    v_x<\frac{g(c-u)}{(c-u)^2+v^2} \quad \text{in } {\Omega}_+.
\end{equation}
Moreover it is also known that $u_q<0$ in $\Omega_+$ c.f. \cite{Con2012}, from which we deduce
\begin{equation}\label{s4eq23}
    v_y>\frac{v}{u-c}v_x\quad\text{in }\Omega_+,
\end{equation}
having used the incompressibility and irrotationality conditions along with relation \eqref{s2eq7}.
Replacing inequality \eqref{s4eq17} in the second member of the Euler equation \eqref{s2eq1}, we conclude
\begin{equation}\label{s4eq24}
    P_y = -g +(c-u)v_x-vv_y<-g+\frac{v_x}{c-u}\left[(c-u)^2+v^2\right]<0\quad\text{in }\Omega_+,
\end{equation}
the last inequality following from relation \eqref{s4eq15}.
Thus the pressure strictly increases with depth at any point of the fluid domain $\overline{\Omega}_+$, while symmetry and periodicity ensure that $P_y<0$ throughout $\overline{\Omega}$.
\end{proof}


\begin{thebibliography}{10}
\expandafter\ifx\csname url\endcsname\relax
  \def\url#1{\texttt{#1}}\fi
\expandafter\ifx\csname urlprefix\endcsname\relax\def\urlprefix{URL }\fi
\expandafter\ifx\csname href\endcsname\relax
  \def\href#1#2{#2} \def\path#1{#1}\fi

\bibitem{Sto1880}
G.~G. Stokes, Considerations relative to the greatest height of oscillatory
  irrotational waves which can be propagated without change in form, in: Math.
  and Phys. Papers, I, Cambridge, 1880, pp. 225--228.

\bibitem{Con2011}
A.~Constantin, Nonlnear Water Waves with Applications to Wave-Current
  Interactions and Tsunamis, no.~81 in CBMS-NSF Regional Conference Series in
  Applied Mathematics, Society for Industrial and Applied Mathematics (SIAM),
  Philadelphia, PA, 2011.

\bibitem{CE2004a}
A.~Constantin, J.~Escher, Symmetry of steady periodic surface water waves with
  vorticity, J. Fluid Mech. 498 (2004) 171--181.

\bibitem{CE2004b}
A.~Constantin, J.~Escher, Symmetry of steady deep-water waves with vorticity,
  Eur. J. of App. Math. 15 (2004) 755--768.

\bibitem{CE2011b}
A.~Constantin, J.~Escher, Analyticity of periodic traveling free surface water
  waves with vorticity, Ann. Math. 173 (2011) 559--568.

\bibitem{Con2012}
A.~Constantin, Particle trajectories in an extreme {S}tokes wave, IMA J. Appl.
  Math. 77 (2012) 293--307.

\bibitem{Tol1996}
J.~F. Toland, Stokes waves, Topol. Meth. Nonlin. Anal. 7.

\bibitem{AFT1982}
C.~J. Amick, L.~E. Fraenkel, J.~F. Toland, On the {S}tokes conjecture for the
  wave of extreme form, Acta Math. 148~(1) (1982) 193--214.

\bibitem{PT2004}
P.~I. Plotnikov, J.~F. Toland, Convexity of stokes waves of extreme form, Arch.
  Rat. Mech. Anal. 171~(3) (2004) 349--416.

\bibitem{Lyo2014}
T.~Lyons, Particle trajectories in extreme {S}tokes wave over infinite depth,
  Discrete Contin. Dyn. Sys. Ser. A 34~(8) (2014) 3095--3107.

\bibitem{Hen2008}
D.~Henry, On the deep-water {S}tokes wave flow, International Mathematics
  Research Notices 2008 (2008) rnn071 (7 pages).

\bibitem{Hen2011}
D.~Henry, Pressure in a deep water {S}tokes wave, J. Math. Fluid Mech. 13~(2)
  (2011) 251--257.

\bibitem{Lyo2015}
T.~Lyons, The pressure in a deep-water {S}tokes wave of greatest height, To
  appear in: J. Math. Fluid Mech.Available at:
  \url{http://arxiv.org/pdf/1508.06819v1}.

\bibitem{Con2006}
A.~Constantin, The trajectories of particles in {S}tokes waves, Invent. Math.
  166~(3) (2006) 523--535.

\bibitem{CE2007}
A.~Constantin, J.~Escher, Particle trajectories in solitary water waves, Bull.
  Amer. Math. Soc. 44~(3) (2007) 423--431.

\bibitem{Hen2006}
D.~Henry, The trajectories of particles in deep-water stokes waves, Int. Math.
  Res. Not. 2006 (2006) 23405.

\bibitem{HKS2008}
K.~A. Haas, A.~B. Kennedy, B.~K. Sapp, Video measurements of large-scale flows
  in a laboratory wave basin, J. Wtrwy, Port, Coast. Oc. Engrg. 134~(1) (2008)
  12--20.

\bibitem{Ume2005}
M.~Umeyama, Reynolds stresses and velocity distributions in a wave-current
  coexisting environment, J. Wtrwy, Port, Coast. Oc. Engrg. 131~(5) (2005)
  203--212.

\bibitem{OVDH2012}
K.~L. Oliveras, V.~Vasan, B.~Deconinck, D.~Henderson, Recovering surface
  elevation from pressure data, SIAM J. Appl. Math 72~(3) (2012) 897--918.

\bibitem{THYLL2005}
C.~H. Tsai, M.~C. Huang, F.~J. Young, Y.~C. Lin, H.~W. Li, On the recovery of
  surface wave by pressure transfer function, Oc. Engrg. 32~(10) (2005)
  1247--1259.

\bibitem{Con2012a}
A.~Constantin, On the recovery of solitary wave profiles from pressure
  measurements, J. Fulid Mech. 669 (2012) 376--384.

\bibitem{CEH2011}
A.~Constantin, J.~Escher, H.~C. Hsu, Pressure beneath a solitary water wave:
  mathematical theory and experiments, Arch. Rat. Mech. Anal. 2011 (2011)
  251--269.

\bibitem{ES2008}
J.~Escher, T.~Schlurmann, On the recovery of the free surface from the pressure
  within periodic traveling water waves, J. of Nonlin. Math. Phys. 15~(sup2)
  (2008) 50--57.

\bibitem{CC2013}
D.~Clamond, A.~Constantin, Recovery of steady periodic wave profiles from
  pressure measurements at the bed, J. Fluid Mech. 714 (2013) 463--475.

\bibitem{Hen2013}
D.~Henry, On the pressure transfer function for solitary water waves with
  vorticity, Math. Ann. 357 (2013) 23--30.

\bibitem{DKD1992}
W.~M. Drennan, K.~K. Kahma, M.~A. Donelan, The velocity field beneath
  wind-waves: observations and inferences, Coast. Engrg. 18~(1) (1992)
  111--136.

\bibitem{Pom2002}
C.~Pommerenke, Conformal maps at the boundary, Handbook of Complex Analysis 1
  (2002) 37--74.

\bibitem{BT2003}
B.~Buffoni, J.~Toland, Analytic Theory of Global Bifurcation: An Introduction,
  Princeton Series in Applied Mathematics, Princeton University Press, 2003.

\bibitem{CS2010}
A.~Constantin, W.~Strauss, Pressure beneath a {S}tokes wave, Comm. Pure App.
  Math. 63~(4) (2010) 533--557.

\bibitem{Ami1987}
C.~J. Amick, Bounds for water waves, Arch. Rat. Mech. Anal. 99~(2) (1987)
  91--114.

\end{thebibliography}
\end{document}